\documentclass[11pt]{amsart}
\usepackage{amssymb,amsthm,amsmath,amstext}
\usepackage{mathrsfs}  
\usepackage{bm}        
\usepackage{mathtools} 
\usepackage{color}
\usepackage[bbgreekl]{mathbbol}
\usepackage{multirow}

\usepackage{cite}

\usepackage[left=1.5in,top=1in,right=1.5in,bottom=1in]{geometry}

\usepackage{graphicx}

\usepackage[all]{xy}
\usepackage{tikz}

\newcommand{\xycenter}[1]{
	\begin{center}
	\mbox{\xymatrix{#1}}
	\end{center}
	}

\theoremstyle{plain}
\newtheorem{theorem}{Theorem}[section]

\newtheorem{proposition}[theorem]{Proposition}
\newtheorem{lemma}[theorem]{Lemma}
\newtheorem{corollary}[theorem]{Corollary}

\theoremstyle{definition}
\newtheorem{definition}[theorem]{Definition}

\theoremstyle{remark}
\newtheorem{remark}[theorem]{Remark}

\newcommand{\sheaf}[1]{\mathscr{#1}}

\newcommand{\OO}{\sheaf{O}}

\newcommand{\MM}{\sheaf{M}}

\newcommand{\C}{\mathbb C}
\renewcommand{\P}{\mathbb P}

\DeclareMathOperator{\codim}{codim}

\newcommand{\Pf}{\mathrm{Pf}}

\newcommand{\wMM}{\widetilde{\MM}}

\newcommand{\sk}{\mathcal{S}}
\newcommand{\skss}{\sk^{ss}}
\newcommand{\skps}{\sk^{ps}}

\newcommand{\wskss}{\widetilde{\sk}^{ss}}

\newcommand{\GL}{\mathrm{GL}}
\newcommand{\GLsix}{\GL_6 (\C )}

\newcommand{\cubicsexplicit}{\P\bigl(H^0\bigl(\OO_{\P^4}(3)\bigr)\bigr)}
\newcommand{\cubics}{\mathcal{C}}

\newcommand{\ring}{R}

\usepackage[backref=page]{hyperref}

\begin{document}

\title[Moduli spaces of $6\times 6$ skew matrices of linear forms]{Moduli spaces of $6\times 6$ skew matrices of linear forms on $\P^4$ with a view towards intermediate Jacobians of cubic threefolds}

\author[B\"ohning]{Christian B\"ohning}\thanks{The first author was supported by the EPSRC New Horizons Grant EP/V047299/1.}
\address{Christian B\"ohning, Mathematics Institute, University of Warwick\\
Coventry CV4 7AL, England}
\email{C.Boehning@warwick.ac.uk}

\author[von Bothmer]{Hans-Christian Graf von Bothmer}
\address{Hans-Christian Graf von Bothmer, Fachbereich Mathematik der Universit\"at Hamburg\\
Bundesstra\ss e 55\\
20146 Hamburg, Germany}
\email{hans.christian.v.bothmer@uni-hamburg.de}

\author[Buhr]{Lukas Buhr}
\address{Lukas Buhr, Institut f\"ur Mathematik\\
Johannes Gutenberg-Universit\"at Mainz\\
Staudingerweg 9\\
55128 Mainz, Germany}
\email{lubuhr@uni-mainz.de}

\date{\today}


\begin{abstract}
It is well known that every smooth cubic threefold is the zero locus of the Pfaffian of a $6\times6$ skew-symmetric matrix of linear forms in $\P^4$. To compactify the space of such Pfaffian representations of a given cubic and to study the construction in families as well as for singular or reducible cubics, it is thus natural to consider the incidence correspondence of Pfaffian representations inside the product of the space of semistable skew-symmetric $6\times 6$ matrices of linear forms in $\P^4$ and the space of cubics. Here we describe concretely the irreducible component of this incidence correspondence dominating the space of skew matrices. 
\end{abstract}

\maketitle

\section{Introduction and background}\label{sIntroduction}

Pfaffian representations of cubic threefolds and associated moduli spaces have been studied by various authors, notably Beauville \cite{Beau00}, \cite{Beau02}, Comaschi \cite{Co20}, \cite{Co21}, Manivel-Mezzetti \cite{MaMe05}, Iliev-Markushevich \cite{IM00}. In this article we seek to prove some general results that are useful to study certain compactifications of such spaces of Pfaffian representations in families. To describe our goal more precisely, we first need to set up some notation and terminology.

\medskip 

Let $\ring$ be the graded polynomial ring over $\C$ in variables $x_0, \dots, x_4$ of weight $1$, and let $\sk = \P \bigl ( (\C^5)^{\vee} \otimes \Lambda^2 \C^6 \bigr)$ be the projective space of skew-symmetric $6\times 6$-matrices with entries linear forms on $\P^4$. The group $G=\GLsix$ acts on $\sk$ by acting trivially on $(\C^5)^{\vee}$ and naturally on $\Lambda^2 \C^6$. This corresponds to performing simultaneous row and column operations on skew-symmetric $6\times 6$-matrices. We denote by
\[
\MM := \skss //G
\]
the good quotient of the locus of semistable points in $\sk$ by the action of $G$. Let
\[
\pi \colon \skss \to \MM
\]
be the canonical projection. We define the subset $\skps \subset \skss$ to be the subset of those points whose orbits in $\skss$ are closed. We recall that every fibre of $\pi$ contains a unique closed orbit, i.e. the orbit of an element in $\skps$. Let $\cubics =\cubicsexplicit$ be the space of cubic hypersurfaces in $\P^4$. 
Consider the incidence correspondence 
\[
\mathcal{T}=\bigl\{  ([M], [F]) \mid \Pf (M) \in (F)\bigr\} \subset \skss \times \cubics 
\]
with its two projections 
\[
\pi_1 \colon \mathcal{T} \to  \skss, \quad \pi_2 \colon \mathcal{T} \to \cubics .
\]
The condition $\Pf (M) \in (F)$ is to be understood as the vanishing of all $2\times 2$ minors of the matrix containing in its rows the coefficients of $\Pf(M)$ and $F$ with respect to some basis of the degree $3$ homogeneous polynomials in $x_0, \dots , x_4$. 

\medskip

Let $\skss_0 \subset \skss$ be the subset consisting of matrices with Pfaffian zero, and let 
\[
\mathcal{T}_0= \pi_1^{-1}(\skss_0) =\skss_0\times \cubics .
\]
Notice that $\pi_1$ is one-to-one onto its image outside of the subset $\mathcal{T}_0 \subset \mathcal{T}$. 
Denote by $\wskss$ the closure of $\pi_1^{-1} (\skss - \skss_0)$ in $\mathcal{T}$. 

The group $G$ acts on $\sk \times \cubics$ if we let it act trivially on $\cubics$. Then $\skss \times \cubics$ is the locus of semistable points for this action, and $\wskss$ is clearly a $G$-invariant  irreducible closed subset of $\skss \times \cubics$. The good categorical quotient $(\skss \times \cubics)//G$ is nothing but $\MM \times \cubics$, and $\wskss$ maps to an irreducible closed subset of $\MM \times \cubics$, which we denote by $\wMM$. 

\medskip

Our goal is to describe the points of $\wskss_0:=\wskss \cap \mathcal{T}_0$ more explicitly, which we will achieve in Theorem \ref{tClosure}. Our initial interest in this problem stems from the fact that $\wMM \to \cubics$ may then be viewed as the \emph{universal family of $G$-equivalence classes of Pfaffian representations of cubic threefolds}, the fibres are projective, and there is a fibre over any cubic, singular, reducible, whatsoever. Also base-changing to curves inside $\cubics$ and discarding components of the resulting total spaces that do not dominate the base curves, we can talk about one-parameter degenerations of (compactified) spaces of $G$-equivalence classes of Pfaffian representations. In a forthcoming article, we will show that the fibre of $\wMM$ over a smooth cubic threefold $X$ coincides with the Maruyama-Druel-Beauville moduli space of equivalence classes of semistable sheaves on $X$ with Chern classes $c_1=0, c_2=2, c_3=0$, or equivalently, the intermediate Jacobian of $X$ blown up in the Fano surface of lines \cite{Beau02}. Thus $\wMM$ gives us a way to construct explicit degenerations of these birational models of the intermediate Jacobians, which, we hope, will pave a way to decide unresolved questions about the cycle theory on such intermediate Jacobians: the most famous perhaps being if, realising these intermediate Jacobians as Prym varieties of double covers of smooth plane quintic curves, half of the class of the Prym curve is an algebraic cohomology class for very general $X$. By \cite{Voi17} a positive answer to this is equivalent to the cubic being Chow zero universally trivial, which in turn is necessary for the cubic to be stably rational. It is a famous open problem if very general cubic threefolds are stably rational or not. 

\medskip

We will make repeated use of the following classification result below. 

\begin{table}
\begin{tabular}{|c|c|c|c| c|}
\hline
& $M$ & $S$ & $Y$ & \\
\hline
 (a)
 &
  $\left(
 \begin{smallmatrix}
       &{l}_{3}&&0&{l}_{0}&{l}_{1}\\
      {-{l}_{3}}&&&{-{l}_{0}}&0&{l}_{2}\\
      &&&{-{l}_{1}}&{-{l}_{2}}&0\\
      0&{l}_{0}&{l}_{1}&&{l}_{4}&\\
      {-{l}_{0}}&0&{l}_{2}&{-{l}_{4}}&&\\
      {-{l}_{1}}&{-{l}_{2}}&0&&&
\end{smallmatrix}
\right)$
&
$\left(
 \begin{smallmatrix}
      {l}_{2}&\\
      {-{l}_{1}}&\\
      {l}_{0}&{l}_{4}\\
      &{l}_{2}\\
      &{-{l}_{1}}\\
      {l}_{3}&{l}_{0}
\end{smallmatrix}
\right)$
&
a smooth conic
&
stable
\\ \hline
 (b)
 &
$\left(
\begin{smallmatrix}
      0&{l}_{0}&{l}_{1}\\
      {-{l}_{0}}&0&{l}_{2}\\
      {-{l}_{1}}&{-{l}_{2}}&0\\
      &&&0&{l}_{2}&{l}_{3}\\
      &&&{-{l}_{2}}&0&{l}_{4}\\
      &&&{-{l}_{3}}&{-{l}_{4}}&0
\end{smallmatrix}
\right)$
&
$\left(
\begin{smallmatrix}
      {l}_{2}&\\
      {-{l}_{1}}&\\
      {l}_{0}&\\
      &{l}_{4}\\
      &{-{l}_{3}}\\
      &{l}_{2}
\end{smallmatrix}
\right)$
 &
 two skew lines
 &
 stable
\\ \hline
 (c)
 &
$\left(
\begin{smallmatrix}
      0&{l}_{0}&{l}_{1}\\
      {-{l}_{0}}&0&{l}_{2}\\
      {-{l}_{1}}&{-{l}_{2}}& 0\\
      &&&0&{l}_{1}&{l}_{2}\\
      &&&{-{l}_{1}}&0&{l}_{3}\\
      &&&{-{l}_{2}}&{-{l}_{3}}&0
\end{smallmatrix}
\right)$
&
$\left(
\begin{smallmatrix}
      {l}_{2}&\\
      {-{l}_{1}}&\\
      {l}_{0}&\\
      &{l}_{3}\\
      &{-{l}_{2}}\\
      &{l}_{1}
\end{smallmatrix}
\right)$
&
 \begin{tabular}{c}
two distinct \\ intersecting lines\\
with an embedded point\\
at the intersection \\
spanning the ambient $\P^4$
\end{tabular}
&
stable
\\ \hline
 (d)
 &
$\left(
\begin{smallmatrix}
&  &  & 0 &l _0 &l_1 \\
 &  &  & -l_0 & 0 & l_2  \\
 &  &  & -l_1 & -l_2 & 0 \\
0 &l _0 &l_1 & & l_3 & l_4\\
-l_0 & 0 & l_2  & -l_3 &  & \\
-l_1 & -l_2 & 0 & -l_4 &  & 
\end{smallmatrix}
\right)$
&
$\left(
\begin{smallmatrix}
      {l}_{2}&\\
      {-{l}_{1}}& -l_4\\
      {l}_{0}&{l}_{3}\\
      &{l}_{2}\\
      &{-{l}_{1}}\\
      &{l}_{0}
\end{smallmatrix}
\right)$
 &
  \begin{tabular}{c}
  a double line lying on \\
a smooth quadric surface\\
\end{tabular}
&
\begin{tabular}{c}
strictly semistable,\\ but not polystable
\end{tabular}
\\ \hline
(e)
 &
$\left(
\begin{smallmatrix}
&  &  & 0 &l _0 &l_1 \\
 &  &  & -l_0 & 0 & l_2  \\
 &  &  & -l_1 & -l_2 & 0 \\
0 &l _0 &l_1 & & l_3 & \\
-l_0 & 0 & l_2  & -l_3 &  & \\
-l_1 & -l_2 & 0 &  &  & 
\end{smallmatrix}
\right)$
&
$\left(
\begin{smallmatrix}
      {l}_{2}&\\
      {-{l}_{1}}& \\
      {l}_{0}&{l}_{3}\\
      &{l}_{2}\\
      &{-{l}_{1}}\\
      &{l}_{0}
\end{smallmatrix}
\right)$
&
  \begin{tabular}{c}
a plane double line \\
with an embedded point,\\
spanning the ambient $\P^4$\\
\end{tabular}
&
\begin{tabular}{c}
strictly semistable,\\ but not polystable
\end{tabular}
\\ \hline
(f) 
&
$\left(
\begin{smallmatrix}
&  &  & 0 &l _0 &l_1 \\
 &  &  & -l_0 & 0 & l_2  \\
 &  &  & -l_1 & -l_2 & 0 \\
0 &l _0 &l_1 & &  & \\
-l_0 & 0 & l_2  &  &  & \\
-l_1 & -l_2 & 0 &  &  & 
\end{smallmatrix}
\right)$
&
$\left(
\begin{smallmatrix}
      {l}_{2}&\\
      {-{l}_{1}}& \\
      {l}_{0}&\\
      &{l}_{2}\\
      &{-{l}_{1}}\\
      &{l}_{0}
\end{smallmatrix}
\right)$
&
\begin{tabular}{c}
 a line \\
together with its \\
full first order\\
infinitesimal \\
neighbourhood
\end{tabular}
&
polystable
\\ \hline
\end{tabular}
\medskip

\caption{Semi-stable matrices $M$ with vanishing Pfaffian}
\label{tPfaffZero}
\end{table}

\begin{theorem}\label{tGeometryM0}
Let $[M]\in \skss$ have vanishing Pfaffian. View $M$ as a map of graded $\ring$-modules
\[
\ring (-1)^{6} \xrightarrow{M} \ring^6.
\]
Let $S$ be a matrix with columns representing a minimal system of generators of the kernel of this map $M$.  Let $Y$ be the rank at most two locus of $M$ with its scheme structure defined by the $4\times 4$ sub-Pfaffians. Then there exists independent linear forms $l_0,\dots,l_4$ and matrices $B \in \mathrm{GL_6}(\C)$ and $B'  \in \mathrm{GL_2}(\C)$ such that after making the replacements 
\begin{align*}
	M & \mapsto B^t M B \\
	S & \mapsto B^{-1} S B'
\end{align*}
we have one of the cases in Table \ref{tPfaffZero}. Moreover, the stability type of $M$ is as described in the last column of Table \ref{tPfaffZero}.
\end{theorem}

\begin{proof}
Most of this is proven in \cite{BB22}, and the results are summarised in Table 2 there. The information about $S$ and the more precise information about $Y$ is readily obtained using Macaulay2, see 
\cite[Table1.m2]{BB-M2}. 
\end{proof}

In the following it will be helpful to notice, that $S$ carries exactly the same information as $M$: 
\begin{proposition}\label{pSameInformation}
Let $M$ and $S$ be matrices as in Table \ref{tPfaffZero}. Then $M$ represents the syzygy module of $S^t$. Furthermore the ideal generated by the $2\times 2$ minors of $S$ is equal to the one generated by the $4 \times 4$ Pfaffians of $M$. 
\end{proposition}

\begin{proof}
In all cases we can compute that the syzygy module of $S^t$ is represented by a not necessarily skew $6 \times 6$ matrix of linear forms. See \cite[\ttfamily{Table1.m2}]{BB-M2}. Since $S^tM^t =0$ we have that the columns of $M^t = -M$ are linear syzygies of $S^t$. Since these columns are linearly independent in all cases, $-M$ and hence $M$ represents the syzygy module of $S^t$. 

The last statement of the Proposition follows by a direct computation done in \cite[\ttfamily{Table1.m2}]{BB-M2}. 
\end{proof}

\section{The main theorem}\label{sStilde}
Here we prove

\begin{theorem}\label{tClosure}
Let $[M]$ be a point in $\skss_0$ and $X = V(F) \subset \P^4$ a cubic threefold. Let $\overline{Y} \subset X$ be the scheme defined by the $4 \times 4$ Pfaffians of $M$ and the cubic polynomial $F$. Then $([M],[F])$ is a point in $\wskss_0$ if and only if $\overline{Y}$ contains a curve of degree $2$. 
\end{theorem}

The proof of this will occupy the rest of the paper. We start by laying the groundwork for some local computations involving jets. 

\begin{remark}\label{rDegree2Curve}
Notice that for matrices $M$ of type $(a) - (e)$ the degree $2$ curve is uniquely defined by $M$. For matrices of type $(f)$ the condition is satisfied for all $F$ such that $V(F)$ contains the rank $2$ locus of $M$ with {\sl reduced} scheme structure, i.e. the line $l_0=l_1=l_2=0$.
\end{remark}

Below by an \emph{algebraic scheme} we mean a scheme separated and of finite type over $\mathrm{Spec}\, (\C)$. 

\begin{definition}\label{dJets}
Let $X$ be an algebraic scheme and $p\in X$ a closed point. Put 
\[
T_n = \mathrm{Spec}\, (\C[t]/(t^{n+1})).
\]
An \emph{$n$-jet starting at $p$, or centred at $p$, in $X$} is a morphism of algebraic schemes $T_n \to X$ mapping the closed point of $T_n$ to $p$. 
We sometimes write the residue class of $t$ in $\C[t]/(t^{n+1})$ as $\epsilon$, and elements $j \in \C[t]/(t^{n+1})$ as
\[
j = j_0 + j_1\epsilon + \dots + j_n\epsilon^n .
\]
Note that if $\mathrm{Spec}\, \bigl( \C[x_1, \dots , x_n]/(f_1 (x_1, \dots , x_n), \dots , f_r (x_1, \dots , x_n)) \bigr)$ is an affine chart on $X$ containing $p$ and $p= (a_1, \dots , a_n)$, the datum of an $n$-jet starting at $p$ is the same as the datum of a solution $(\bar{a}_1, \dots , \bar{a}_n)$ of the equations $f_1=\dots =f_r =0$ where $\bar{a}_i \in \C[t]/(t^{n+1})$ is a lift of $a_i$.
\end{definition}

\begin{definition}\label{dOperationsOnJets}
The natural ring homomorphism
\[
\C[t]/(t^{n+1}) \to  \C[t]/(t^{m+1}) 
\]
for $m \le n$ induces a corresponding operation on jets: given an $n$-jet $j\colon T_{n} \to X$ centred at a point $p$ in an algebraic scheme $X$, we get an $m$-jet 
\[
\tau_{\le m} (j) \colon T_m \to X,
\]
called the \emph{truncation} of $j$ in degrees $\le m$.

\medskip

Moreover, the ring homomorphism
\begin{align*}
\C[t]/(t^{n+1}) & \to  \C[s]/(s^{rn+1}) \\
\bar{t} & \mapsto \bar{s}^r
\end{align*}
induces an operation on jets: given an $n$-jet $j\colon T_{n} \to X$ centred at a point $p$ in an algebraic scheme $X$, we get an $rn$-jet 
\[
\gamma_r (j) \colon T_{rn}\to X
\]
called the \emph{$r$-fold covering} of $j$.
\end{definition}

\begin{proposition}\label{pTangentS}
Let $[M]$ be a point in $\skss_0$. Then the codimension of the tangent space $T_{[M]} (\skss_0)$ in the ambient $T_{[M]}(\skss)$ is given by the second column in Table \ref{tTangentS}. 
\end{proposition}

\begin{proof}
Let $M' = \left( \sum_{k=0}^4 a_{ijk}x_k \right)_{1\le i,j\le 6}$ be a general skew matrix of linear forms. Then we consider a $1$-jet $M+\epsilon M'$ at $M$ in $\skss$. Then we get 
\[
\Pf (M +\epsilon M') = \epsilon F' 
\]
and $M+\epsilon M'$ is a $1$-jet in $\skss_0$ if and only if $F' =0$. The coefficients of $F'$ are linear in the $a_{ijk}$, so we obtain a set of linear equations on the $a_{ijk}$ whose rank is the codimension of the tangent space. The computation of the rank is done in \cite[Table2.m2]{BB-M2}. 
\end{proof}

Consider the tangent cone of an algebraic scheme $X$ at the point $p$. Choose an affine chart as above with $p= (a_1, \dots , a_n) =(0, \dots , 0)$. 
The tangent cone $TC_p (X)$ is the subscheme of $\C^n$ given as the zero locus of the leading terms of all elements in the ideal $I_X=(f_1, \dots , f_r)$. 

\

If $I$ is the ideal of the tangent cone $TC_p(X)$ and $I_{\le 2}$ is the degree at most $2$ part of this ideal, we call the
vanishing locus $V(I_{\le 2}) =: \bigl( TC_p(X) \bigr)_2$ the {\sl degree two approximation of the tangent cone}.

The tangent cone of a variety in a given point is often difficult to compute, but 
the degree $2$ approximation still has some useful computational properties. In particular a version of Hensel lifting still holds:

\begin{proposition} \label{pHensel}
Let $X \subset \C^n$ be a variety, $p \in X$ a point. Then
$p' \in \C^n $ is a point in the degree $2$ approximation of the tangent cone if and only if there exists a point $p'' \in \C^n$ such that the $2$-jet
\[
	j := p + \epsilon p' + \epsilon^2 p''
\]
is contained in $X$.
\end{proposition}

\begin{proof}
Without restriction we can assume $p=0$. Let $f_1,\dots,f_m \in I_X$ be polynomials whose initial terms generate $I_{\le 2}$. We can assume that these have expansions
\begin{align*}
	f_1 &:=  f_1' +  f_1'' + h.o.t. \\
	\vdots & \quad \quad \vdots \\
	f_k &:=  f_k' +  f_k'' + h.o.t. \\
	f_{k+1} &:= \quad \quad f_{k+1}'' + h.o.t. \\
	\vdots & \quad \quad \quad \quad \quad \vdots \\
	f_{m} &:= \quad \quad f_{m}'' + h.o.t. \\
\end{align*}
with $f_1' \dots f_k'$ linearly independent linear forms, and $f_{1}'',\dots, f_{m}''$ homogeneous quadratic polynomials, of which $f_{k+1},\dots,f_{m}$ are linearly independent. 

Now a $2$-jet $j = p+\epsilon p' + \epsilon^2 p''$ lies on $X$ if and only if $f_i(j) = 0$ for $i=1,\dots,m$.  If $p=0$ and $p'$ is in the degree $2$ approximation of the tangent cone, then this system of equations reduces to
\[
	0=f_i(j) = f_i'(p'') + f_i''(p')		\quad \quad i=1,\dots,k.
\]
For given $p'$ this is a set of linear equations for $p''$. Since $f_1',\dots,f_k'$ are linearly independent, this set of linear equations has a solution.

Let conversely $j=p+\epsilon p' + \epsilon p''$ be a $2$-jet on $X$. Evaluating the first $k$ equations in $j$ and considering the terms linear in $\epsilon$, we get that $f_1'(p')=  \dots = f_k'(p') = 0$. Evaluating the remaining equations in $j$ gives $f_{k+1}''(p') = \dots = f_m''(p') = 0$. So $p'$ lies on the degree $2$ approximation of the tangent cone.
\end{proof}

\begin{remark}\label{rHensel}
In the situation of the Proposition above the proof also shows that we can compute the degree $2$ approximation of the tangent cone in the folowing way. 

\begin{enumerate}
\item[(1)] Consider a $2$-jet
\[
	j = p + \epsilon p' + \epsilon^2 p''
\]
with $p'$ in $T_p X$ general and $p'' \in \C^n$ general. 
\item[(2)] Evaluate the generators of $I_X$ in $j$. This yields a vector space $V$ of polynomials quadratic in $p'$ and linear in $p''$. 
\item[(3)] The generators of $I_2$ are those elements of $V$ whose linear part vanishes. This can be computed by solving a linear system of equations.
\end{enumerate}
\end{remark}

\

\begin{proposition}\label{pTangentCone1}
Let $[M]$ be a point in $\skss_0$. Let $I$ be the ideal of the tangent cone $TC_{[M]} (\skss_0) \subset T_{[M]} (\skss )$. Then the degree $2$ part $I_2$ is the same as the degree $2$ part of the ideal listed in the third column of Table \ref{tTangentS}. 
\end{proposition}

\begin{proof}
Let $M' = \left( \sum_{k=0}^4 a_{ijk}x_k \right)_{1\le i,j\le 6}$ and $M''=  \left( \sum_{k=0}^4 b_{ijk}x_k \right)_{1\le i,j\le 6}$ be general skew matrices of linear forms. We consider a $2$-jet $M+\epsilon M'+\epsilon^2 M''$ at $M$ in $\skss$ and compute 
\[
\Pf (M +\epsilon M'+\epsilon^2 M'' ) = \epsilon F' +\epsilon^2 F''.
\]
We choose $M'$ such that $F'=0$. As explained in the proof of Proposition \ref{pTangentS} this means that $M'$ represents a tangent vector to $\skss_0$ at $[M]$. The coefficients of $F''$ are homogeneous of degree two where $\deg a_{ijk}=1$ and $\deg b_{ijk}=2$. The quadrics in the table are those linear combinations of the coefficients of $F''$ that no longer contain any of the $b_{ijk}$. These are contained in the ideal of the tangent cone by Remark \ref{rHensel}. Finding those linear combinations amounts to solving a linear system of equations and is done in \cite[Table2.m2]{BB-M2}. 
\end{proof}

\begin{table}
\begin{tabular}{|c|c|c|c| c|}
\hline 
 type of $[M]$     & $\codim T_{[M]} (\skss_0)$  & ideal of $TC_{[M]} (\skss_0)$ in $T_{[M]} (\skss_0)$  \\
\hline 
                  (a)  & 28  & $(0)$  \\ \hline
                  (b)  & 27 & $(0)$ \\ \hline
                  (c) & 26 &
                  $({a}_{124}-{a}_{354},{a}_{024}-{a}_{344})\cdot
                  ({a}_{054})$
                   \\ \hline
                  (d) & 27 & $(0)$ \\ \hline
                  (e) & 26 & $(a_{454}, a_{354})\cdot(a_{014})$ \\ \hline
                  & & $\mathrm{minors}_{2\times2} \bigl(N_1|N_2\bigr) \cap \mathrm{minors}_{2\times2}\begin{pmatrix} N_1 \\ N_2 \end{pmatrix}$ 
                  \\
                  & &with \\
                  (f)& 22&$N_1 = \begin{pmatrix}
      {a}_{124}&{a}_{454}\\
      {{a}_{024}}&{a}_{354}\\
      {a}_{014}&{a}_{344}
      \end{pmatrix}
      $
      \\
       & &\\
      &&$
      N_2=\begin{pmatrix}
      {a}_{123}&{a}_{453}\\
      {{a}_{023}}&{a}_{353}\\
      {a}_{013}&{a}_{343}
      \end{pmatrix}
      $                  
                  \\
                  \hline
\end{tabular}
\medskip
\caption{
}
\label{tTangentS}
\end{table}


\begin{proposition}
We have the following stratification of $\skss_0$:
\xycenter{
 (b) \ar[r] \ar[dr]& (d) \ar[dr]\\
(a) \ar[r] & (c) \ar[r] & (e) \ar[r]& (f) 
}
where $\xymatrix{(x) \ar[r] & (y)}$ means that the stratum of type $(y)$ matrices lies in the closure of the stratum of type $(x)$ matrices.
\end{proposition}

\begin{proof}
We deal with the closure relation symbolised by each arrow separately.

\medskip

\noindent \fbox{\textbf{$(b) \to (c)$}} Notice that matrices of type $(b)$ are conjugate to matrices of the form
\[
\left(
\begin{smallmatrix}
      0&{l}_{0}&{l}_{1}\\
      {-{l}_{0}}&0&{l}_{2}\\
      {-{l}_{1}}&{-{l}_{2}}&0\\
      &&&0&{l}_{4}&{l}_{2}\\
      &&&{-{l}_{4}}&0&{l}_{3}\\
      &&&{-{l}_{2}}&{-{l}_{3}}&0
\end{smallmatrix}
\right)
\]
Here we use that $\Lambda^2 \C^3 \simeq (\C^3)^{\vee}$ as $\mathrm{GL}_3 \, \C$-representation, whence passing to a conjugate matrix we can replace the entries of the bottom right $3\times 3$ skew matrix by any invertible linear combination of them. Now specialising $l_4 \to l_1$, we get a matrix of type $(c)$.

\medskip

\noindent \fbox{\textbf{$(d) \to (e)$}} This is obvious letting $l_4\to 0$. 

\medskip

\noindent \fbox{\textbf{$(e) \to (f)$}} This is obvious letting $l_3\to 0$. 

\medskip

\noindent \fbox{\textbf{$(a) \to (c)$}} Consider the family of skew $6\times 6$ matrices 
\[
M_t = \begin{pmatrix}
     &{l}_{3}&&0&{l}_{0}&{l}_{1}\\
     {-{l}_{3}}&&&{-{l}_{0}}&0&{l}_{2}\\
     &&&{-{l}_{1}}&{-{l}_{2}}&0\\
     0&{l}_{0}&{l}_{1}&&{l}_{3}+{l}_{4}t&\\
     {-{l}_{0}}&0&{l}_{2}&-{l}_{3}-{l}_{4}t&&\\
     {-{l}_{1}}&{-{l}_{2}}&0&&&\end{pmatrix}
\]
One has $\Pf (M_t)=0$ and the rank $2$ locus with its reduced structure is defined by the ideal: 
\[
\left({{l}_{1}},{{l}_{2}},{l}_{0}^{2}-{l}_{
     3}^{2}-{l}_{3}{l}_{4}t\right)
\]
For $t\neq 0$ this defines a smooth conic, while for $t = 0$ it defines two distinct intersecting lines. Therefore $M_t$ is of type $(a)$ for $t\not=0$ and of type $(c)$ for $t=0$.
\medskip

\noindent \fbox{\textbf{$(b) \to (d)$}} Consider the family of skew $6\times 6$ matrices 
\[
M_t = \begin{pmatrix}
      &{l}_{3}t^{2}&{l}_{4}t^{2}&0&{l}_{0}&{l}_{1}\\
     {-{l}_{3}t^{2}}&&&{-{l}_{0}}&0&{l}_{2}\\
     {-{l}_{4}t^{2}}&&&{-{l}_{1}}&{-{l}_{2}}&0\\
     0&{l}_{0}&{l}_{1}&&{l}_{3}&{l}_{4}\\
     {-{l}_{0}}&0&{l}_{2}&{-{l}_{3}}&&\\
     {-{l}_{1}}&{-{l}_{2}}&0&{-{l}_{4}}&&\end{pmatrix}
\]
One can check that $\Pf (M_t)=0$ and the rank $2$ locus with its reduced scheme structure is
defined by
\[
	(l_2,l_1+t l_4,l_0+tl_3) \cap (l_2,l_1-t l_4,l_0-t l_3)
\]
For $t\neq 0$ this defines two skew lines and hence $M_t$ is of type $(b)$ for $t\not=0$. Also $M_0$ is of type $(d)$. 

\medskip

\noindent \fbox{\textbf{$(c) \to (e)$}} Consider the family of skew $6\times 6$ matrices 
\[
M_t = \begin{pmatrix}
      &{l}_{3}t^{2}&&0&{l}_{0}&{l}_{1}\\
      {-{l}_{3}t^{2}}&&&{-{l}_{0}}&0&{l}_{2}\\
      &&&{-{l}_{1}}&{-{l}_{2}}&0\\
      0&{l}_{0}&{l}_{1}&&{l}_{3}&\\
      {-{l}_{0}}&0&{l}_{2}&{-{l}_{3}}&&\\
      {-{l}_{1}}&{-{l}_{2}}&0&&&\end{pmatrix}
\]
One can check that $\Pf (M_t)=0$ and the rank $2$ locus with its reduced structure is: 
\[
\left({l}_{1},{l}_{2},{l}_{0}^{2}-{l}_{3}^{2
      }t^{2}\right)
\]
For $t\neq 0$ this defines two distinct, but intersecting lines and hence $M_t$ is of type $(c)$. Moreover $M_0$ is of type $(e)$. 

\noindent
All computations necessary for the above are done in \cite[deformations6x6.m2]{BB-M2}.

\medskip

\noindent
Now notice that the rank $\le 2$ locus of a specialization of $M$ must contain the specialization of 
the rank $\le 2$ locus of $M$. This shows that only those specialisations depicted in the diagram can exist. For example, neither can two skew lines be deformed into a subscheme of a smooth conic nor a smooth conic into a subscheme of two skew lines. Also notice that $(a)$ does not specialise to $(d)$ since a plane conics can only specialise to plane double lines, which are not contained in any smooth quadric. 
\end{proof}

\begin{definition}\label{dAB}
We denote by $\mathcal{A}$ and $\mathcal{B}$ the closures of the loci of matrices $[M]$ of type $(a)$ and $(b)$ in $\skss_0$. 
\end{definition}

\begin{proposition}\label{pGeomskss}
Both $\mathcal{A}$ and $\mathcal{B}$ are irreducible and $\skss_0$ is their union. Moreover, 
\[
\codim (\mathcal{A} \subset \skss) = 28, 
\quad \codim (\mathcal{B} \subset \skss) = 27,
\quad \codim (\mathcal{A} \cap \mathcal{B} \subset \skss) = 29.
\]
\end{proposition}

\begin{proof}
That $\mathcal{A}$ and $\mathcal{B}$ are irreducible is clear. Then one can compute the codimension by noticing that matrices of type $(a)$ and $(b)$ form one orbit under the action of $\mathrm{GL}_6 (\C)\times \mathrm{GL}_5 (\C )$ (the latter acting by changing coordinates $x_0, \dots , x_4$). Computing the dimension of the stabiliser of a representative in $(a)$ or $(b)$ yields the result. 
\end{proof}

\begin{corollary}\label{cTangentCone}
Let $[M]$ be a point in $\skss_0$. Then the ideal $I$ of the tangent cone $TC_{[M]} (\skss_0) \subset T_{[M]} (\skss_0)$ is listed in the third column of Table \ref{tTangentS}. 
\end{corollary}

\begin{proof}
The tangent space of $\mathcal{A}$ in a point of type $(a)$ is of codimension $28$ which is also the codimension of $\mathcal{A}$, so $\mathcal{A}$ is smooth in these points and the tangent cone has ideal $(0)$. The same argument also explains the entries in the third column for points of type $(b)$ and $(d)$. 

Points $[M]$ of type $(c)$, $(e)$ and $(f)$ lie in the intersection of $\mathcal{A}$ and $\mathcal{B}$. We have seen in Proposition \ref{pTangentCone1} that the quadrics listed in the third column of Table \ref{tTangentS} vanish on the tangent cone  $TC_{[M]} (\skss_0) \subset T_{[M]} (\skss )$. In all three cases the quadrics listed in the third column cut out the union of two irreducible reduced varieties of codimension $28$ and $27$ in $T_{[M]}(\skss)$. Therefore this union must be equal to 
the tangent cone at these points. 
\end{proof}

\begin{lemma}\label{lFirstNonZero}
Let $p=([M], [F])$ be a point in $\wskss$, and let 
\[
j =( [M+ \epsilon M_1 + \dots + \epsilon^n M_n], [F+ \epsilon F_1 + \dots + \epsilon^n F_n])
\]
be a jet centred at $p$ and contained in $\wskss$. Suppose that
\[
\Pf ( M+ \epsilon M_1 + \dots + \epsilon^n M_n ) =\epsilon^n G
\]
with $G\neq 0$. Then $F$ and $G$ are nonzero scalar multiples of each other. 
\end{lemma}

\begin{proof}
The defining equations of $\mathcal{T}$ imply that $\Pf ( M+ \epsilon M_1 + \dots + \epsilon^n M_n )$ is a multiple of $F+ \epsilon F_1 + \dots + \epsilon^n F_n$ with a factor of proportionality in $\C[t]/(t^{n+1})$. Therefore $G$ must be proportional to $F$ with a nonzero scalar in $\C$.
\end{proof}

We now start showing one direction of Theorem \ref{tClosure}.

\begin{proposition}\label{pOneDirection}
Let $p= ([M],[F])$ be a point in $\wskss_0$ and let $\overline{Y}\subset X$ be the scheme defined by the $4 \times 4$ Pfaffians of $M$ and the cubic polynomial $F$.  
Then $\overline{Y}$ contains a curve of degree $2$. 
\end{proposition}

\begin{proof}

We divide the proof into cases, depending on whether $[M]$ is of type (a)-(f).

\medskip

\noindent \fbox{\textbf{Case 1: $[M]$ of type (a), (b) or (d)}} In these cases $\mathcal{T}_0$ is smooth in $p$ because of Corollary \ref{cTangentCone}. In such a point there must be a tangent vector to $\mathcal{T}$ that is not tangent to $\mathcal{T}_0$, or in terms of jets, a $1$-jet centred at $p$ in $\mathcal{T}$:
\[
j = ([M + \epsilon M_1], [F + \epsilon F_1])
\]
such that 
\[
\Pf (M + \epsilon M_1 ) = \Pf (M) + \epsilon G = \epsilon G
\]
with $G \neq 0$. Now $G$ is a linear combination of the $4\times 4$ sub-Pfaffians of $M$ with coefficients entries of $M_1$ by Pfaffian Laplace expansion \cite[equation (D.1) p. 116]{FP98} : hence in these cases $X=V(G)$ contains the scheme defined by the $4 \times 4$ Pfaffians of $M$.

\medskip

\noindent \fbox{\textbf{Case 2: $[M]$ of type (c) or (e)}} If the tangent space to $\mathcal{T}$ at $p$ is strictly larger than the tangent space to $\mathcal{T}_0$ at $p$, then we can argue as in Step 1, so we assume the two tangent spaces are equal in the following. Therefore the tangent cones $TC_p (\mathcal{T})$ and  $TC_p (\mathcal{T}_0)$ can be viewed as living in the same ambient space, and we have $TC_p (\mathcal{T}_0) \subsetneq TC_p (\mathcal{T})$: let $I_{\mathcal{T}}, I_{\mathcal{T}_0}$ be the ideals of $\mathcal{T}, \mathcal{T}_0$ in an affine neighbourhood of $p$ which we can assume to be the origin. Then $I_{\mathcal{T}}\subsetneq I_{\mathcal{T}_0}$. Choose a Gr\"obner basis $B$ for $I_{\mathcal{T}}$ with respect to some monomial ordering refining the order given by total degree. Let $f\in I_{\mathcal{T}_0}$ be a polynomial which is not in $I_{\mathcal{T}}$. The reduction $\bar{f}$ of this polynomial modulo the Gr\"obner basis $B$ is nonzero and the initial term of $\bar{f}$ (i.e., the lowest degree homogeneous component) is not in the ideal of the tangent cone $TC_p (\mathcal{T})$.


In cases (c), (e) the ideal of the tangent cone $TC_p (\mathcal{T}_0)$ is given by $(xz, yz)$ with suitable choice of coordinates $x, y, z$. The ideal of $TC_p (\mathcal{T})$ contains at most one quadric $(\alpha x +\beta y) z$ in the ideal $(xz, yz)$. Therefore the support of the degree $2$ approximation of $TC_p(\mathcal{T})$ is strictly bigger than the support of the degree $2$ approximation of $TC_p(\mathcal{T}_0)$. 
By Proposition \ref{pHensel} this implies that there is a $2$-jet $p+\epsilon p' + \epsilon^2 p''$ contained in $\mathcal{T}$ with $p'$ in the tangent space to $\mathcal{T}$ at $p$, which is equal to the tangent space to $\mathcal{T}_0$ at $p$, which is not contained in $\mathcal{T}_0$. Now we can check by computer algebra \cite[Table1.m2]{BB-M2} that for all $2$-jets 
\[
	j = ([M + \epsilon M_1 + \epsilon^2 M_2], [F + \epsilon F_1+ \epsilon^2 F_2]])
\]
with $M$ a point of type $(c)$ or $(e)$ and
\[
	\Pf (M + \epsilon M_1  + \epsilon^2 M_2) = \epsilon^2 G
\]
we have that $G$ vanishes on the subscheme defined by the $4\times 4$ Pfaffians of $M$ with embedded points removed.

\medskip

\noindent \fbox{\textbf{Case 3: $[M]$ of type (f)}} 
There exists a pair of formal power series
\[
 (\overline{M}_t , \overline{F}_t) =([M +  M_1 t+ M_2 t^2 + \dots ], [F + F_1 t+ F_2 t^2 + \dots ])
\]
such that 
\[
\Pf (\overline{M}_t ) \equiv t^{n} G \mod t^{n+1}
\]
with $n\in \mathbb{N}$ and $G$ a nonzero cubic polynomial. Let $x \in \C^5\backslash \{0\}$ be a point such that $M(x)$ has rank $\le 2$; for type (f) this is equivalent to requiring that $M(x)=0$. We claim that there exists a positive integer $r$ and a power series 
\[
\overline{x}_s = x + x_1s + x_2 s^2 + \dots 
\]
such that the $4\times 4$ sub-Pfaffians of 
\[
 \overline{M}_{s^r} (\overline{x}_s) 
\]
vanish modulo $s^{rn}$. Geometrically, we assume that $ (\overline{M}_t , \overline{F}_t)$ lies on $\mathcal{T}_0$ modulo $t^n$, and we claim that after a ramified covering that amounts to replacing $t$ by $s^r$, we can deform the point $x$ along with $M$ such that it still lies in the rank $\le 2$ locus modulo $t^n =s^{rn}$. Granting the claim for the moment, we can finish the proof as follows. 

On the one hand, we have 
\[
\Pf \bigl(  \overline{M}_{s^r} (\overline{x}_s)  \bigr) \equiv s^{rn} G(\overline{x}_s) \equiv s^{rn} G(x) \mod s^{rn+1}.
\]
On the other hand, for any matrix $N$
\[
\Pf \bigl(  N  \bigr) = \sum_{\alpha < \beta} \pm l_{\alpha\beta}(N) q_{\alpha\beta} (N)
\]
where $l_{\alpha\beta}(N)$ is the $(\alpha, \beta)$-entry of $N$ and $q_{\alpha\beta} (N)$ the sub-Pfaffian of the matrix obtained from $N$ by erasing rows and columns $\alpha, \beta$. Applying this to $N=\overline{M}_{s^r} (\overline{x}_s) $ we see that 
\[
\Pf \bigl(  \overline{M}_{s^r} (\overline{x}_s)  \bigr) =  \sum_{\alpha < \beta} \pm l_{\alpha\beta}\bigl(\overline{M}_{s^r} (\overline{x}_s) \bigr) q_{\alpha\beta} \bigl( \overline{M}_{s^r} (\overline{x}_s) \bigr).
\]
Since
\[
l_{\alpha\beta}\bigl(\overline{M}_{s^r} (\overline{x}_s) \bigr) \equiv 0 \mod s
\]
and
\[
q_{\alpha\beta} \bigl( \overline{M}_{s^r} (\overline{x}_s) \bigr) \equiv 0 \mod s^{rn}
\]
we get that 
\[
s^{rn} G(x) \equiv \Pf \bigl(  \overline{M}_{s^r} (\overline{x}_s)  \bigr) \equiv 0 \mod s^{rn+1}
\]
whence $G(x)=0$. Therefore $G$ vanishes on the support of the rank $\le 2$ locus of $M$, which by Remark \ref{rDegree2Curve} is enough.  

\medskip

It therefore remains to prove the above claim. Let $\mathcal{B}=\mathrm{Spec}\, \C [t]/(t^n)$. We can view $\overline{M}_t$ as a section of $\skss_0\times_{\C} \mathcal{B}$ over $\mathcal{B}$. We consider the subscheme $\mathcal{Y}$ in $
\P^4\times_{\C} \mathcal{B}$ defined by the $4\times 4$ sub-Pfaffians of $\overline{M}_t$. By our previous classification, all fibres of $\mathcal{Y}$ over $\mathcal{B}$ are curves possibly with embedded points or multiple structure. We consider $\mathcal{Y}'\subset \mathcal{Y}$ an irreducible dominant component of relative dimension $1$ over $\mathcal{B}$. Cutting down by a relative hyperplane in $\P^4 \times_{\C} \mathcal{B}$ containing $x$ we can obtain an irreducible curve $\mathcal{C}\to \mathcal{B}$ of rank $\le 2$ points of $\overline{M}_t$. After a base change $s=t^r$, the pull-back of $\mathcal{C}$ will acquire a section. 
\end{proof}

It remains to show the other direction of Theorem \ref{tClosure}.

\begin{proposition}\label{pOtherDirection}
Let $p=([M], [F])$ be a point in $\mathcal{T}_0$ such that the scheme $\overline{Y}$ defined by the $4 \times 4$ Pfaffians of $M$ and $F$ contains a curve of degree $2$. Then $p\in \wskss_0=\wskss \cap \mathcal{T}_0$.
\end{proposition}

\begin{proof}
We again split up the argument into cases according to the type of $M$. 

\medskip

\noindent \fbox{\textbf{Case 1: $[M]$ of type (a), (b), (d)}} In these cases, the subscheme $Y$ defined by the $4\times4$ sub-Pfaffians of $M$ is a pure dimensional degree $2$ curve, hence by the assumption $F$ can be written 
\[
F = \sum_{\alpha< \beta} \pm l_{\alpha \beta} q_{\alpha\beta} 
\]
where $q_{\alpha\beta} $ is the sub-Pfaffian of the matrix obtained from $M$ by erasing rows and columns $\alpha, \beta$, and the $l_{\alpha \beta}$ are some linear forms. We can now collect the $l_{\alpha\beta}$ (with appropriate signs) into an antisymmetric matrix $M_1$ such that 
\[
\Pf (M +\epsilon M_1) = \epsilon F . 
\]
This means that there exists a $1$-jet centred at $p=([M], [F])$ that is contained in $\mathcal{T}$ but not in $\mathcal{T}_0$. This can only happen if $p\in \wskss \cap \mathcal{T}_0$. 

\medskip

\noindent \fbox{\textbf{Case 2: $[M]$ of type (c), (f) and $X$ smooth}} In this case, by our assumption, the smooth cubic threefold $X=V(F)$ contains two intersecting and possibly identical lines. Since we already know that all pairs $([M'], [F'])$ with $M'$ of type $(b)$ and rank $2$ locus of $M'$ contained in $X'=V(F')$ are in $\wskss \cap \mathcal{T}_0$, it suffices to show that our given $([M], [F])$ is a limit of such $([M'], [F'])$. In our argument we take $F'=F$. By standard facts about lines on smooth cubic threefolds, each pair of such lines is the limit of a family of pairs of skew lines on the cubic. We can assume that for $t\neq0$ this family of skew lines is defined by a family of matrices of type $(b)$:
\[
\begin{pmatrix}
A_t & 0\\
0 & B_t
\end{pmatrix}.
\]
For $t=0$ we get a matrix of type $(c)$ or $(f)$ by our construction. 

\medskip

\medskip

\noindent \fbox{\textbf{Case 3: $[M]$ of type (e) and $X$ smooth}} Consider a $6 \times 6$ skew matrix of linear forms $M$ of type $(e)$, which we write as
\[
	M = \begin{pmatrix}
		0 & 2A \\
		2A & -B
		\end{pmatrix}
\]
with $A, B$ skew, and $X = V(F)$ a smooth cubic threefold containing the unique double line $\overline{Y}$ that is contained in the rank $2$ locus of $M$. Our goal is to write down a family of matrices $([M_t], [F])$ in $\wskss_0$ where $M$ is of type $(b)$ for $t\neq 0$ and $M=M_0$. 
Intuitively, the strategy will be to deform the reduced underlying scheme $\ell$ of $\overline{Y}$ in the direction of the double structure. 

The entries of $A$ define $\ell$. We claim that $M$ still has rank $2$ on the vanishing locus of the entries of $A+\epsilon B$ with $\epsilon^2 =0$: indeed, the matrix
\[
	M = \begin{pmatrix}
		0 & -2\epsilon B \\
		-2\epsilon B & -B
		\end{pmatrix}
\]
has rank $2$. Also it follows that there is a family of lines $\ell_t$ on $X$ with $\ell_0=\ell$ and tangent direction at $t=0$ defined by the vanishing of the entries of $A+\epsilon B=0$. We can also choose the family such that the general $\ell_t$ does not intersect $\ell$. 
Therefore there exists a family of skew-symmetric $3 \times 3$ matrices of linear forms defining the family of lines $\ell_t$  on $X$ such that locally we have
\[
	A_t = A + t B + \text{higher order terms}
\]
For $t\not=0$ consider the family of skew-symmetric $6 \times 6$ matrices
\[
	\begin{pmatrix}
	A& 0 \\
	0 & -A_t
	\end{pmatrix}.
\]
Conjugating with $\left( \begin{smallmatrix} 1 & -1 \\ 1 & 1 \end{smallmatrix} \right)$ we obtain
\[
	\begin{pmatrix}
	A- A_{t}& A+ A_{t}\\
	A+ A_{t} & A - A_{t}
	\end{pmatrix}.
\]
Conjugating with $\left( \begin{smallmatrix} t & 0 \\ 0 & 1 \end{smallmatrix} \right)$ gives
\[
 \begin{pmatrix}
	t^2(A-A_{t}) & t(A + A_{t}) \\
	t(A+A_{t}) & (A-A_{t})
	\end{pmatrix}.
\]
Now notice that $A-A_{t} = -t(B + \text{h.o.t.})$.
Scaling with $t^{-1}$ we obtain a family
\[
M_t =  \begin{pmatrix}
	-t^2(B + \text{h.o.t}) & A + A_{t} \\
	A+A_{t} & -(B + \text{h.o.t})
	\end{pmatrix}
\]
that by construction is conjugate to
\[
	t^{-1}\begin{pmatrix}
	A& 0 \\
	0 & -A_{t}
	\end{pmatrix}
\]
for $t \not=0$. For 
$t=0$ we have
\[
M_0 = \begin{pmatrix}
	0& 2A\\
	2 A & -B
	\end{pmatrix} = M.
\]
This accomplishes what we wanted. 

\medskip

\noindent \fbox{\textbf{Case 4: $[M]$ of type (c), (e) (f) and $X$ singular}} In this case, we are given a singular cubic threefold $X=V(F)$ containing a subscheme $Z$ which is either a pair of intersecting but distinct lines (type $(c)$), or a plane double line (type $(e)$), or a reduced line (type $(f)$). We observe that $X$ can be written as a limit of smooth cubic threefolds $X_t$ containing the same $Z$. Therefore the pair $([M], [F])$ is in the limit of a family of pairs $([M], [F_t])$ where $X_t = V(F_t)$ for $t\neq 0$ is smooth and $([M], [F_t])$ is in $\wskss \cap \mathcal{T}_0$ by Case 2 and 3. 
\end{proof}

\providecommand{\bysame}{\leavevmode\hbox to3em{\hrulefill}\thinspace}
\providecommand{\href}[2]{#2}

\end{document}